\documentclass[12pt]{article}  % `article' with A4 paper size option
\usepackage{amsmath,amssymb,color,amscd}

\def\seq#1#2#3{#1_{#2},\,\ldots,#1_{#3}}
\def\w{\widetilde}
\def\h{\widehat}
\def\VV{{\underline{V}}}
\def\vv{{\underline{v}}}

\def\tt{{\underline{t}}}

\def\ww{\underline{w}}

\def\1{\underline{1}}

\def\P{\Bbb P}

\def\Z{\Bbb Z}

\def\C{\Bbb C}

\def\OO{{\cal O}}

\newtheorem{theorem}{Theorem}
\newtheorem{statement}{Statement}

\newenvironment{definition}
{\smallskip\noindent{\bf Definition\/}:}{\smallskip\par}
\newenvironment{example}
{\smallskip\noindent{\bf Example\/}.}{\smallskip\par}

\newenvironment{remark}
{\smallskip\noindent{\bf Remark\/}.}{\smallskip\par}
\newenvironment{remarks}
{\smallskip\noindent{\bf Remarks\/}.}{\smallskip\par}
\newenvironment{proof}
{\noindent{\bf Proof\/}.}{{ $\Box$}\smallskip\par}

\title{Equivariant Poincar\'e series of filtrations
\footnote{Math. Subject Class. 14B05, 16W70,
16W22. Keywords: finite group actions,
filtrations, Poincar\'e series.} }

\author{
A.~Campillo
\and F.~Delgado \and S.M.~Gusein-Zade
\thanks{
Partially supported by the grant MTM2007-64704 (with the help of
FEDER Program). Third author
is also partially supported by the grants RFBR-10-01-00678,
NSh-709.2008.1 and Visitantes distinguidos en la UCM -- Grupo
SANTANDER (Ay16/09).
} }

\date{}
\begin{document}
\def\eps{\varepsilon}

\maketitle

\begin{abstract}
We offer a new approach to a definition of an equivariant
version of the Poincar\'e series. This Poincar\'e series is defined
not as a power series, but as an element of the Grothendieck ring
of $G$-sets with an additional structure.
We compute this Poincar\'e
series for natural filtrations on the ring of germs of functions
on the plane $(\C^2,0)$ with a finite group representation.
\end{abstract}

\section*{Introduction}\label{sec0}
The notion of the Poincar\'e series of a multi-index filtration
(say on the ring of germs of functions on a variety) was
introduced in \cite{CDK}. It can be expressed as an integral
with respect to the Euler characteristic (duly defined) over the
projectivization of the space of germs of functions (see e.g.
\cite{IJM}).
For some natural multi-index filtrations on the ring of germs
of functions the Poincar\'e series are related with topological
invariants of singularities. For example the Poincar\'e series of
the multi-index filtration on the ring $\OO_{\C^2,0}$ of germs
of functions in two variables defined by orders of a function on
branches of a plane curve singularity tends to coincide with the
Alexander polynomial (in several variables) of the corresponding
link: see e.g. \cite{IJM}. An equivariant version of the
Poincar\'e
series (for an action of a finite group $G$) was defined
in \cite{MMJ}.
For universal Abelian covers of rational  surface singularities
the equivariant Poincar\'e series was computed in \cite{FAA}.
A. N\'emethi has found that the equivariant Poincar\'e
series of Abelian
covers of surface singularities are connected with computation of
Seiberg--Witten invariants of their links: \cite{Nemethi}. The
equivariant Poincar\'e series in \cite{MMJ} is defined as an
integral with respect to the Euler characteristic
over the projectivization of the union of subspaces of
equivariant functions (with respect to one-dimensional
representations of the group). Thus it takes into account only
one-dimensional representations ignoring all others. Therefore it
seems to be better adjusted for Abelian groups. Nevertheless even
in this case it ignores some important information. For example
it does not define, in general, the usual (non-equivariant)
Poincar\'e series.

Here we offer a new approach to a definition of an equivariant
version of the Poincar\'e series. This Poincar\'e series is defined
not as a power series, but as an element of the Grothendieck ring
of $G$-sets with an additional structure. (This Grothendieck
ring, in
a natural way, is isomorphic to the series ring $\Z[[\seq t1r]]$
for the case of the trivial group.)
The idea to define equivariant analogues of (numerical)
invariants as elements of the Grothendieck ring of $G$-sets was
used, for example, in \cite{L-R, G-W} where an equivariant
analogue of the Lefschetz number was defined as an element of
the Burnside ring which coincides with the Grothendieck ring of
finite $G$-sets.

We compute the equivariant Poincar\'e
series for natural filtrations on the ring of germs of functions
on the plane $(\C^2,0)$ with a finite group representation.
In this cases the expressions for the equivariant Poincar\'{e} series
are of A'Campo type, i.e. this Poincar\'{e} series are
products/ratios of binomials of the form $(1-U)$ for some
``irreducible" elements $U$.

%%%%%%%%%%%%%%%%%%%%%%%%%%%
\section{Equivariant Poincar\'{e} series}
%%%%%%%%%%%%%%%%%%%%%%%%%%%

Let $(V,0)$ be a germ of a complex analytic variety with an
action of a finite group $G$. The group
$G$ acts on the ring $\OO_{V,0}$ of germs of functions on
$(V,0)$: $a^* f(x)= f(a^{-1}x)$ ($f\in \OO_{V,0}$, $a\in G$,
$x\in V$). A function
$v: \OO_{V,0}\to \Z_{\ge 0} \cup \{+\infty\}$ is called an
{\it order function} if $v(\lambda f)= v(f)$ for a non-zero
$\lambda\in\C$ and $v(f_1+f_2)\ge \min \{v(f_1), v(f_2)\}$.
(If besides that $v(f_1 f_2)=v(f_1)+v(f_2)$, the function $v$ is
a valuation.) A multi-index filtration of the ring $\OO_{V,0}$
is defined by a collection $\seq v1r$ of order functions:
\begin{equation}\label{eq1}
J(\vv) = \{ f\in \OO_{V,0} : \vv(f)\ge \vv\}\; ,
\end{equation}
where $\vv=(\seq v1r)\in \Z^{r}_{\ge 0}$,
$\vv(f)= (v_1(f), \ldots, v_r(f))$ and 
$(\seq{v'}1r)\ge (\seq{v''}1r)$ if and only if
$v'_i\ge v''_i$ for all $i=1,\ldots, r$.

Let us assume that the filtration $J(\vv)$ is {\it finitely
determined}. This means that, for any $\vv\in \Z^{r}_{\ge 0}$,
there exists an integer $k$ such that
${\mathfrak m}^k\subset J(\vv)$ where ${\mathfrak m}$ is the maximal
ideal in $\OO_{V,0}$. Let $\P\OO_{V,0}$ be the projectivization
of the ring $\OO_{V,0}$. For $k\ge 0$, let
$J^k_{V,0} = \OO_{V,0}/{\mathfrak m}^{k+1}$ be the space of $k$-jets
of functions. It is a finite-dimensional vector space. Let
$\P J^k_{V,0}$ be the projectivization of the jet space and let
$\P^* J^k_{V,0} = \P J^k_{V,0} \cup \{*\}$. Here one can say that $\P^* J^k_{V,0}$
is the factor space of $J^k_{V,0}$ by the standar $\C^*$-action:
the added point $*$ corresponds to the orbit of the origin. There
are natural maps $\pi_k: \P \OO_{V,0}\to \P^* J^k_{V,0}$ and
$\pi_{k,\ell}: \P J^k_{V,0}\to \P^* J^{\ell}_{V,0}$ for
$k\ge\ell$. A subset
$B\subset \P \OO_{V,0}$ is called {\it cylindric} if there exists
$k$
and a constructible subset $X\subset \P J^k_{V,0}$ ($\subset \P^*
J^k_{V,0}$) such that $B = \pi_k^{-1}(X)$. The Euler
characteristic $\chi(B)$ of a cylindric subset
$B = \pi_k^{-1}(X)$ is defined as $\chi(X)$. A function $\psi$ on
$\P \OO_{V,0}$ with values in an Abelian group $A$ is called
{\it cylindric} if, for any $a\in A$, $a \neq 0$, $\psi^{-1}(a)$
is cylindric. The integral of a cylindric function
$\psi : \P \OO_{V,0} \to A$ with respect to the Euler
characteristic is defined as
$$
\int\limits_{\P\OO_{V,0}} \psi d \chi = \sum_{a\in A, a \neq 0}
\chi(\psi^{-1}(a)) a
$$
(if the right hand side makes sense in the group $A$; otherwise
the function $\psi$ is not integrable).

In the same way one can define cylindric sets, their Euler
characteristics and the integral with respect to the Euler
characteristic for the factor $\P\OO_{V,0}/G$ of the space
$\P\OO_{V,0}$ by an action of a finite group $G$ which respects
the filtration by the powers of the maximal ideal.

The (usual) Poincar\'{e} series of the multi-index filtration
(\ref{eq1}) can be defined as
$$
P_{\{v_i\}} (\seq t1r) = \int\limits_{ \P \OO_{V,0}}
\tt^{\,\vv(f)} d\chi\; ,
$$
where $\tt=(\seq t1r)$,
$\tt^{\,\vv}= t_1^{v_1}\cdot\ldots \cdot t_r^{v_r}$;
$\tt^{\,\vv(f)}$
is considered as a function on $\P \OO_{V,0}$ with values in
the ring (Abelian group) $\Z[[\seq t1r]]$, $t^{+\infty}$ is assumed to be equal to zero.

We shall define an equivariant version of the Poincar\'{e} series of
the filtration (\ref{eq1}) as an element of the Grothendieck ring
of sets with an additional structure.

\begin{definition}
A (``locally finite") $(G,r)-\mbox{set}$ $A$ is a
triple $(X^A, \ww^A, \alpha^A)$ where
\begin{itemize}
\item $X^A$ is a $G$-set, i.e. a set with a $G$-action;
\item $\ww^A$ is a function on $X^A$ with values in
$\Z^r_{\ge 0}$;
\item $\alpha^A$ associates to each point $x\in X^A$ a
one-dimensional representation $\alpha^A_{x}$ of the isotropy
subgroup $G_x= \{ a\in G : ax =x\}$ of the point $x$;
\end{itemize}
satisfying the following conditions:
\begin{enumerate}
\item[1)] $\alpha^A_{ax}=a \alpha^A_x a^{-1}$ for $x\in X^A$,
$a\in G$;
\item[2)] for any $\ww\in \Z^r_{\ge 0}$ the set
$\{x\in X^A: \ww^A(x)\le \ww\}$ is finite.
\end{enumerate}
\end{definition}

All (locally finite) $(G,r)$-sets form an Abelian semigroup in
which the sum is defined as the disjoint union. Let the Cartesian
product $A\times B$
of two $(G,r)$-sets $A$ and $B$ be the triple where
$X^{A\times B}=X^A\times X^B$ with the natural action of the
group
$G$,
$\ww^{A\times B}(x,y) =\ww^A(x)+\ww^{B}(y)$,
$\alpha^{A\times B}_{(x,y)} =
\alpha^A_{x}\cdot \alpha^B_{y}$. This makes the semigroup of
$(G,r)$-sets a semiring. The unit 1 is represented by the
one-point set $X^1=\{*\}$ with $\ww^{1}(*)=0$ and
$\alpha^1_{*}=1$.

Let
$K_0((G,r)-\mbox{sets})$
be the Grothendieck ring of
(locally finite) $(G,r)$-sets, i.e. the Grothendieck group
corresponding to the semigroup described above with the
multiplication defined by the Cartesian product.

\begin{example}
For the trivial group $G=\{e\}$, the ring
$K_0((G,r)-\mbox{sets})$ is isomorphic to the ring
$\Z[[\seq t1r]]$ of power series where $t_i$ is
the one-point set $X^{t_i}=\{*\}$ with $\ww^{t_i}(*)=
(0,\ldots, 1,\ldots,0)$, $1$ is at the $i$-th place.
\end{example}

\begin{remarks}

{\bf 1.}
Let $K_0'$ be the subring of the Grothendieck ring $K_0((G,r)-\mbox{sets})$
generated by elements with the trivial action of the group $G$.
Let $R(G)$ be the ring of representations of the group $G$ and
let $R_1(G)$ be its subring generated by one-dimensional
representations. The ring $K'_{0}$ in a natural way is isomorphic
to the ring $R_1(G)[[\seq t1r]]$ of power series with
coefficients in the ring $R_1(G)$. This is just the ring where
the equivariant Poincar\'{e} series defined in \cite{MMJ} lives.

{\bf 2.}
For an arbitrary group $G$, one can choose a system of elements
which generates (after completion) the ring $K_0((G,r)-\mbox{sets})$.
This identifies the Grothendieck ring $K_0((G,r)-\mbox{sets})$
with the ring of series in several variables modulo certain
relations. However this is somewhat artificial.
\end{remarks}

\begin{example}
Let $G=\Z_2$ and let $r=1$. A natural set of (multiplicative) generators of the
Grothendieck ring $K_0((G,r)-\mbox{sets})$ consists of the
following elements:
\begin{enumerate}
\item
$t_1$: $X^{t_1}=\Z_2$, the function $w^{t_1}$ takes two values
$0$ and $1$ (the isotropy subgroup of a point is trivial and thus
$\alpha^{t_1}_x=1$);
\item
$t_2$: $X^{t_2}=\Z_2$, the function $w^{t_2}$ takes value $0$ at
each point;
\item
$t_3$: $X^{t_3}= \{*\}=\Z_2/\Z_2$, $w^{t_3}(*)=1$,
$\alpha^{t_3}_{*}=1$;
\item
$t_4$:
$X^{t_4} = \{*\}=\Z_2/\Z_2$, $w^{t_4}(*)=0$,
$\alpha^{t_4}_{*}=-1$;
\end{enumerate}
One can see that $t_4^2=1$, $t_2^2=2t_2$, $t_2t_4=t_2$, $t_1t_4=t_1$, $t_1t_2=2t_1$ and
$K_0((G,r)-\mbox{sets}) = \Z[[t_1, t_2, t_3,
t_4]]/\langle t_4^2 -1, t_2^2-2t_2, t_2t_4-t_2, t_1t_4-t_1, t_1t_2-2t_1 \rangle$.
Pay attention that natural additive generators of $K_0((G,r)-\mbox{sets})$ like $A=A_{m_1m_2}$: $X^A=\Z_2$, $w^A$ takes the values $m_1$ and $m_2$ ($\alpha^A_{x}=1$ for each $x$) have somewhat complicated expressions in terms of the multiplicative generators $t_i$, $i=1, \ldots, 4$. For example $A_{15}=t_1^4 t_3+t_2t_3^3-4t_1^2t_3^2$.
The ring $K_0((G,r)-\mbox{sets})$ can be considered as a completion of the ring of regular functions (polynomials) on the union $\C^2\cup \C^1\cup \C^1$ (the variety defined by the ideal $\langle t_4^2 -1, t_2^2-2t_2, t_2t_4-t_2, t_1t_4-t_1, t_1t_2-2t_1 \rangle$).
\end{example}

Let again the filtration $\{J(\vv)\}$ be defined by the order functions $\seq v1r$
by (\ref{eq1}). For an element $f\in \P\OO_{V,0}$, that is for a
function germ defined up to a constant factor,
let $G_f$ be the isotropy subgroup of the corresponding point of
$\P\OO_{V,0}$:
$G_f =\{a \in G : a^*f =\lambda_f(a)f\}$. The map $a\mapsto\lambda_f(a)$ defines a
one-dimensional representation $\lambda_f$ of the subgroup $G_f$. (The representation
$\lambda_f$ is defined by the class of $f$ in the projectivization $\P\OO_{V,0}$.)
For an element $f\in \P\OO_{V,0}$
(or rather for its orbit $Gf$), let $T_{f}$ be the element of
the Grothendieck ring $K_0((G,r)-\mbox{sets})$ represented by the orbit $Gf$ of $f$
(as a $G$-set) with  $\ww^{T_f}(a^*f)=\vv(a^*f)$
and $\alpha^{T_f}_{a^*f}=\lambda_{a^*f}$ ($a\in G$).
One should exclude those $f$ for which $\ww^{T_f}(a^*f)\notin \Z^r_{\ge 0}$, i.e. $w_i^{T_f}(a^*f)=+\infty$ for some $a\in G$. One can formally define $T_f$ to be equal to zero in this case. (This is an analogue of the assumption $t^{+\infty}=0$ in the definition of the usual Poincar\'e series.)

Let us consider $T: f\mapsto T_f$ as a function on
$\P \OO_{V,0}/G$ with values in the Grothendieck ring
$K_0((G,r)-\mbox{sets})$.
One can see that this function is
cylindric. This follows from the condition that the
filtration $\{J(\vv)\}$ is finitely determined. Moreover, the function $T$ is
obviously integrable (with respect to the Euler characteristic).

\begin{definition}
The equivariant Poincar\'{e} series $P^G_{\{v_i\}}$ of the filtration $\{J(\vv)\}$ is
defined by
$$
P^G_{\{v_i\}} = \int\limits_{\P\OO_{V,0}/G} T_f d\chi
\in K_0((G,r)-\mbox{sets})
\; .
$$
\end{definition}

\begin{example}
Let $F_G\subset \P\OO_{V,0}$ be the set of fixed points of the action of the group
$G$ on $\P \OO_{V,0}$. One has $F_G/G \cong F_G$ and $F_G$
consists of classes of
functions equivariant with respect to the $G$-action: $a^*f=\lambda_f(a)f$.
For $f\in F_G$, $T_f\in R_1(G)[[\seq t1r]]$ (see Remark 1 above).
One can see that the
integral $\int\limits_{F_G} T_f d \chi$ as an element of $R_1(G)[[\seq t1r]]$
coincides with the equivariant Poincar\'{e} series
$P^G_{\{v_i\}}(\seq t1r)$ in the sense of \cite{MMJ}.
\end{example}

\begin{statement}\label{st1}
The equivariant Poincar\'{e} series $P^G_{\{v_i\}}$ determines
the equivariant Poincar\'{e} series
$P^G_{\{v_i\}}(\seq t1r)$ in the sense of \cite{MMJ}.
\end{statement}

\begin{proof}
One can see that the Grothendieck ring $K_0((G,r)-\mbox{sets})$
is a semidirect product of two subrings
$K'_0 \simeq R_1(G)[[\seq t1r]]$ and $K''_0$ generated by
elements
represented by $G$-sets with the trivial action of the group $G$
and with actions without fixed
points respectively. Let
$\Pi': K_0((G,r)-\mbox{sets})\to K'_0$ be the projection onto the
first summand. One can see that
$P^G_{\{v_i\}}(\seq t1r)=\Pi'(P^G_{\{v_i\}})$.
\end{proof}

Suppose that all the order functions $v_i$ are such that $v_i(f)=+\infty$ if and only if $f=0$. This takes place, for example, if all $v_i$ are divisorial valuations.

\begin{statement}\label{st2}
In the described situation the equivariant Poincar\'{e} series  $P^G_{\{v_i\}}$ determines
the (usual) Poincar\'{e} series
$P_{\{v_i\}}(\seq t1r)\in \Z[[\seq t1r]]$.
\end{statement}

\begin{proof}
Let $\Pi: K_0((G,r)-\mbox{sets}) \to \Z[[\seq t1r]]$ be the homomorphism which sends
an element $A$ of $K_0((G,r)-\mbox{sets})$  to
$\sum\limits_{x\in X^A} \tt^{\ww^A(x)}$. The Fubini formula (see
e.g. \cite{IJM}) applied
to the factorization map $p: \P\OO_{V,0} \to \P\OO_{V,0}/G$ gives
\begin{equation}\label{eqn9}
P_{\{v_i\}}(\seq t1r) = \int\limits_{\P\OO_{V,0}}\tt^{\,\vv}
d\chi =
\int\limits_{\P\OO_{V,0}/G} \Pi(T_f) d \chi = \Pi(
P^G_{\{v_i\}})\;
.
\end{equation}
\end{proof}

\begin{remark}
Another situation when Equation~\ref{eqn9} (and thus Statement~\ref{st2}) holds is when the set $\{v_i\}$ is $G$-invariant, i.e. contains all shifts $v_{ig}$ of its elements $v_i$ for $g\in G$: $v_{ig}(f):=v_{i}((g^{-1})^*f)$.
\end{remark}

%%%%%%%%%%%%%%%%%%%%%%%%%%%%%%%%
\section{The equivariant Poincar\'{e} series for filtrations on
$\OO_{\C^2,0}$}
%%%%%%%%%%%%%%%%%%%%%%%%%%%%%%%%

Let $\pi: ({\cal{X, D}})\to (\C^2,0)$ be a modification of the
plane $(\C^2,0)$ by a
$G$-invariant sequence of blowing-ups. This means that the action of $G$ on the plane
$\C^2$ lifts to an action on the smooth surface ${\cal X}$ and
the exceptional divisor
${\cal D}= \pi^{-1}(0)$ (a normal crossing divisor on ${\cal X}$)
is invariant with respect to the
$G$-action.
At each intersection point $x$ of two components of the
exceptional divisor ${\cal D}$, each of these components is invariant with respect to
the isotropy subgroup $G_x =\{a\in G : a x =x\}$ of the  point $x$. All components
$E_{\sigma}$ ($\sigma\in \Gamma$) of the exceptional divisor ${\cal D}$ are
isomorphic to the complex projective line $\C\P^1$.
Let ${\stackrel{\bullet}{E}}_\sigma$ be the ``smooth part" of the
component $E_\sigma$, i.e. $E_\sigma$ itself minus intersection points with
all other components of the exceptional divisor ${\cal D}$, let
${\stackrel{\bullet}{\cal
D}}=\bigcup\limits_\sigma{\stackrel{\bullet}{E}}_\sigma$
be the smooth part of the exceptional divisor ${\cal D}$, and let
$\widehat{\cal D}={\stackrel{\bullet}{\cal D}}/G$ be the
corresponding factor space,
i.e. the space of orbits of the action of the group $G$ on
${\stackrel{\bullet}{\cal D}}$.
Let $p: {\stackrel{\bullet}{\cal D}} \to \h{\cal D}$ be the
factorization map.

For $x\in {\stackrel{\bullet}{\cal D}}$, let $\w{L_x}$ be a germ
of a smooth curve on ${\cal X}$ transversal to
${\stackrel{\bullet}{\cal D}}$ at the point $x$ and invariant
with respect to the isotropy subgroup $G_x$ of the point $x$. The
image $L_x=\pi(\w{L_x})\subset (\C^2,0)$ is called a {\it
curvette} at the point $x$. Let the curvette $L_x$ be given by an
equation $h'_x=0$, $h'_x\in\OO_{\C^2,0}$. Let
$h_x = \sum\limits_{a\in
G_x}{\dfrac{h'_x}{a^*h'_x}}(0)\cdot a^*
h'_x$. The germ $h_x$ is $G_x$-equivariant and $\{h_x=0\}=L_x$.
Moreover, in what follows we assume that the germ $h_x$ is fixed
this way for one point $x$ of each $G$-orbit and is defined by
$h_{ax} = a h_x a^{-1}$ for other points of the orbit.

Let $\{\Xi\}$ be a stratification of the space (in fact of a
smooth curve) $\widehat{\cal D}$ (${\widehat{\cal
D}}=\coprod \Xi$) such that:
\begin{enumerate}
\item[1)] each stratum $\Xi$ is connected;
\item[2)] for each point $\h{x}\in \Xi$ and for each point $x$
from
its preimage $p^{-1}(\h{x})$, the conjugacy class of the
isotropy
subgroup $G_x$ of the point $x$ is the same, i.e. depends only on
the stratum $\Xi$.
\end{enumerate}
The last is equivalent to say that, over each
stratum $\Xi$, the map
$p: {\stackrel{\bullet}{\cal D}} \to \h{\cal D}$ is a covering.

For a component $E_{\sigma}$ of the exceptional divisor
${\cal D}$, let $v_{\sigma}$ be the corresponding divisorial
valuation on the ring $\OO_{\C^2,0}$: for $f\in \OO_{\C^2,0}$,
$v_\sigma(f)$ is the order of zero of the lifting $f\circ\pi$ of
the function $f$ along the component $E_\sigma$. Let
$\{1,\ldots, r\}$ be a subset of $\Gamma$, and let $\seq v1r$ be
the corresponding divisorial valuations. They define the
multi-index filtration (\ref{eq1}).

For a point $x\in {\stackrel{\bullet}{\cal D}}$, let $T_x$ be the
element of the Grothendieck ring $K_0((G,r)-\mbox{sets})$ defined
by $T_x = T_{h_x}$ where $h_x$ is a $G_x$-equivariant function
defining a curvette at the point $x$.
The element $T_x$ is well-defined. i.e. does not depend on the
choice of the function $h_x$. One can see that the
element $T_x$ is one and the same for all points from the
preimage of a stratum $\Xi$ and therefore it will be denoted by
$T_{\Xi}$.

\begin{theorem}
\begin{equation}\label{eq2}
P^G_{\{v_i\}} = \prod\limits_{\{\Xi\}} (1-T_{\Xi})^{-\chi(\Xi)}
\; .
\end{equation}
\end{theorem}

\begin{proof}
Let $Y$ be the configuration space of effective divisors on
${\stackrel{\bullet}{\cal D}}$. The space $Y$ has the natural
representation of the form
$$
Y\ = \coprod\limits_{\{k_\sigma\}} \left(
\prod\limits_{\sigma\in\Gamma}  S^{k_\sigma}
{\stackrel{\bullet}{E_\sigma}} \right) =
\prod\limits_{\sigma\in \Gamma}\left(\coprod\limits_{k=0}^\infty
S^k {\stackrel{\bullet}{E_\sigma}} \right)\; ,
$$
where $S^kZ = Z^k/S_k$ is the $k$-th symmetric power of the space
$Z$. There is the natural action of the group $G$ on the space
$Y$. Let $\h{Y}=Y/G$ be the space of $G$-orbits on $Y$.

For a point $y \in Y$, $y=\sum\limits_{i=1}^n x_i$, let $T_y$ be
the element of the Grothendieck ring $K_0((G,r)-\mbox{sets})$ defined by
$T(y) = T_{\prod h_{x_i}}$. This way one has a $G$-invariant map
$T: Y \to K_0((G,r)-\mbox{sets})$ and therefore a map
$\h T: \h Y \to K_0((G,r)-\mbox{sets})$.

Let $\w{Y}$ be the configuration space of effective divisors on the (smooth) curve
$\h{\cal D} = {\stackrel{\bullet}{\cal D}}/G$. The space $\w{Y}$ has the natural
representation (corresponding to the stratification $\{\Xi\}$) of the form
$$
\w{Y} = \coprod\limits_{\{k_\Xi\}} \left(\prod\limits_{\Xi} S^{k_{\Xi}} \Xi \right) =
\prod\limits_{\Xi}\left(\coprod\limits_{k=0}^\infty
S^k \Xi \right).
$$
Let $\w{T}: \w{Y}\to K_0((G,r)-\mbox{sets})$ be the function on
$\w Y$ defined by
$$
\w{T}(\w y) = \prod\limits_{\Xi} T_{\Xi}^{\, k_{\Xi}}
$$
for $\w{y}\in \prod\limits_{\Xi}S^{k_{\Xi}} \Xi$.

There is the natural map $q: \h Y\to \w Y$ which sends the orbit of a point
$y=\sum\limits_{i=1}^m x_i\in Y$ to the point
$\w y=\sum\limits_{i=1}^m \w{x_i}$, where $\w{x_i}=p(x_i)$ is the
orbit of the point
$x_i$. The preimage of a point
$\w y= \sum\limits_{i=1}^m \w{x}_i\in \w{Y}$ in $Y$ (i.e. under
the map $q\circ p$)
is the product of the orbits $Gx_i$, $p(x_i)=\w{x_i}$, with the
natural action of the group $G$. Moreover
$\w{T}(\w{y})=\sum\limits_{\h{y}\in q^{-1}(\w{y})} \h{T}(\h{y})$.

Therefore
$$
\int\limits_{\h{Y}} \h{T}d \chi =
\int\limits_{\w{Y}} \w{T}d \chi\; .
$$

One has
$$
\int\limits_{\widetilde Y} {\widetilde T}  d\chi
=\prod\limits_{\Xi}\left(\sum\limits_{k=0}^{\infty}\chi(S^k\Xi)T_{\Xi}^k\right)\,.
$$
Using the well-known equation
$$
\sum\limits_{k=0}^\infty \chi(S^kZ)t^k = (1-t)^{-\chi(Z)}\,,
$$
one gets
$$
\int\limits_{\widetilde Y} {\widetilde T} d\chi
=\prod\limits_{\Xi}\left(1-T_{\Xi}\right)^{-\chi(\Xi)}\,.
$$

Let us fix an arbitrary $\VV\in\Z_{\ge 0}^r$ and let us prove
Equation (\ref{eq2}) up to elements in the ring
$K_0((G,r)-\mbox{sets})$ represented by triples $(X^A, \ww^A,
\alpha^A)$ with $\ww^A(x)\le \VV$ for all $x\in X^A$.
This will imply Equation (\ref{eq2}) itself. For that we can
suppose that the modification
$\pi:({\cal X,D})\to (\C^2,0)$ is such that, for any $f\in
\OO_{\C^2,0}$ with $\vv(f)\le \VV$, the strict
transform of the curve $\{f=0\}$ intersects the exceptional
divisor ${\cal D}$ only at smooth points.
This can be achieved by an additional $G$-invariant series of
blowing-ups of intersection points of components of the exceptional divisor.
Such additional blowing-ups  add, to the stratification
$\{\Xi\}$ of ${\stackrel{\bullet}{\cal D}}$, strata with zero
Euler characteristics
and therefore do not change the right hand side of Equation
(\ref{eq2}). Let  $\OO_{\C^2,0}^{\VV}$ be the set of
$f\in \OO_{\C^2,0}$ with $\vv(a^*f)\le \VV$ for all $a\in G$.

Let $I:\P\OO_{\C^2,0}^{\VV}\to Y$ be the map which sends the
class of a function $f\in\OO_{\C^2,0}^{\VV}$ to the intersection
of the strict transform  of the zero-level curve $\{f=0\}$ with
the exceptional divisor ${\cal D}$ (i.e. to the collection of the
intersection points
counted with the corresponding multiplicities). One has the commutative diagram
$$
\begin{CD}
\P\OO^{\VV}_{\C^2,0} @>I>> Y \\
@VpVV @VVpV \\
\P\OO^{\VV}_{\C^2,0}/G @>\h{I}>> \h{Y}
\end{CD}
$$

In general, $T\ne\widehat T\circ \widehat I$ because the
isotropy subgroup of a point in $\P\OO_{\C^2,0}^{\VV}$ can be a
proper  subgroup
of the  isotropy subgroup of its image in $Y$ and therefore the
$G$-orbits of a point  in $\P\OO_{\C^2,0}^{\VV}$ and of its
image in $Y$ can be different
(as $G$-sets). (If these isotropy subgroups coincide, one has the equality.)

To compute the integral
of the function $T$ over
${\P\OO_{\C^2,0}^{\VV}/G}$,
we shall, for each point $\widehat{y}\in\widehat{Y}$, construct a point
$f_{\widehat{y}}$ in $\h{I}^{-1}(\widehat{y})$ (i.e. the orbit of a
function) so that $T_{f_{\widehat{y}}}=\widehat T(\widehat y)$ and the complement
$\h I^{-1}(\widehat{y})\setminus\{f_{\widehat{y}}\}$ is fibred
into $\C^*$-families with the function $T$ constant along fibres. This implies that
$$
\int\limits_{\h{I}^{-1}(\widehat{y})}{T}d\chi = \widehat T(\widehat y)
$$
and the Fubini formula applied to the map $\widehat I$ gives (up to
terms under consideration)
$$
\int\limits_{\P\OO_{\C^2,0}^{\VV}/G}Td\chi =
\int\limits_{\widehat{Y}}
{\widehat{T}}d\chi =
\prod\limits_{\Xi}\left(1-T_{\Xi}\right)^{-\chi(\Xi)}\,.
$$

Let  $\widehat{y}\in\widehat{Y}$ be the orbit of
$y=\sum\limits_{i=1}^{m} x_i \in Y$. Let
$f_y:=\prod\limits_{i=1}^{m} h_{x_i}$. The isotropy
subgroup of $f_y$ in $\P\OO_{\C^2,0}$ coincides with the isotropy
subgroup of $y$ and therefore $T_{f_{\widehat{y}}}={\widehat{T}}(\widehat{y})$
(here $f_{\widehat{y}}=p(f_{y})$ is the image of  $f_y$ in $\P\OO_{\C^2,0}/G$).

Let $g\in I^{-1}(y)$.  The strict transforms of the curves
$\{g=0\}$ and  $\{f_y=0\}$ intersect the exceptional divisor
${\cal D}$ at the same points
with the same multiplicities. Therefore the ration
$\psi=\dfrac{g\circ \pi}{f_y\circ \pi}$ of the liftings ${g\circ
\pi}$ and ${f_y\circ \pi}$ of the functions
$g$ and $f_y$ to the space ${\cal X}$ of the modification has
neither
zeros no poles on the exceptional divisor ${\cal D}$ and thus is
constant on it.
Therefore (multiplying $g$ by a constant) one can assume that
the ratio $\psi$ is equal to $1$ on ${\cal D}$.

Let $g_{\lambda}:= f_y+ \lambda(g-f_y)$, $\lambda\in\C^*$. One
has $\dfrac{g_{\lambda}\circ \pi}{f_y\circ \pi}=1$ on the
exceptional divisor ${\cal D}$.
Therefore $I(g_{\lambda})=I(f_{y})=y$. Moreover the isotropy
subgroup of each $g_{\lambda}$ in  $\P\OO_{\C^2,0}$ coincides with
the isotropy subgroup of $g$
and thereore $T_{p(g_{\lambda})}$ is constant. This proves the
statement.
\end{proof}

As above, let
$\C^2$ be the complex plane with a representation of a finite
group $G$. Let $(C,0)\subset (\C^2,0)$ be a (generally
speaking reducible)
plane curve singularity and let $C=\bigcup\limits_{i=1}^{r}C_i$
be its representation as the union of irreducible components.
Let $\varphi_i:(\C, 0)\to(\C^2, 0)$ be a parameterization
(uniformization) of the component $C_i$, i.e., a germ of an
analytic map such that ${\rm{Im}}\,\varphi_i=C_i$ and $\varphi_i$ is an
isomorphism between $\C$ and $C_i$ outside of the origin.
For a germ $g\in{\OO}_{\C^2, 0}$, let $v_i(g)$ be
the degree of the leading term in the power series
decomposition of the function $g\circ\varphi_i:(\C,0)\to \C$
at the origin:
$$
g\circ\varphi_i(\tau)=
a\cdot\tau^{v_i(f)}+\ \mbox{terms~of~higher~degree} ,
\ \ a\ne 0\; .
$$
(If $g\circ\varphi_i\equiv 0$, one assumes $v_i(g)=+\infty$.)
The function $v_i: {\OO}_{\C^2, 0}\to \Z_{\ge 0}\cup \{+\infty\}$ is a valuation.
The collection $v_1$, \dots, $v_r$ of the valuations corresponding to the components of the curve $C$ defines a multi-index
filtration $\{J(\vv)\}$ by (\ref{eq1}).

Let $\pi:({\cal X,D})\to(\C^2,0)$ be a $G$-invariant embedded
resolution
of the curve $C$. This means that ${\cal X}$ is a smooth complex
surface with an action of the group $G$
commuting with $\pi$,  $\pi$ is a proper complex analytic map
which is an isomorphism outside of the exceptional divisor
${\cal D}=\pi^{-1}(0)$, the total transform $\pi^{-1}(C)$ of the
curve
$C$ is a normal crossing divisor on ${\cal X}$
(and therefore ${\cal D}$ is a normal crossing divisor as well),
and moreover, at each intersection point $x$
of two components of the
total transform $\pi^{-1}(C)$ of the curve $C$, each of these
components is invariant with respect to the isotropy subgroup
$G_x=\{a\in G: ax=x\}$ of the point $x$. This resolution can be obtained from  $(\C^2,0)$ by a ($G$-invariant)
sequence of blowing-ups at points in the preimage of the origin.

Let ${\stackrel{\circ}{\cal D}}$ be the "smooth" $G$-invariant part
of
the exceptional divisor ${\cal D}$ in the total transform
$\pi^{-1}(C)$
of the curve $C$, i.e. ${\cal D}$ itself minus all intersection
points
of all the components of  $\pi^{-1}(C)$ and of its images under the $G$-action.
Let $\{\Xi\}$ be a stratification of
${\stackrel{\circ}{\cal D}}/G$
defined in the same way as above (for divisorial valuations).
(In fact one can simply
take the intersections of the strata of the stratification for
the divisorial case with ${\stackrel{\circ}{\cal D}}/G$.) Let
$T_{\Xi}\in K_0((G,r)-\mbox{sets})$
be defined as above. The same arguments give the following statement.

\begin{theorem}
$$
P^G_{\{v_i\}}=\prod\limits_{\Xi}\left(1-T_{\Xi}\right)^{-\chi(\Xi)}\,.
$$
\end{theorem}

\end{document}